\documentclass[a4paper,english]{article}
\pdfoutput=1
\usepackage[T1]{fontenc}
\usepackage[utf8]{inputenc}
\usepackage{babel}
\usepackage{amsthm}
\usepackage{amsmath}
\usepackage{amssymb}
\usepackage{esint}
\usepackage[unicode=true,pdfusetitle,
 bookmarks=true,bookmarksnumbered=false,bookmarksopen=false,
 breaklinks=false,pdfborder={0 0 1},backref=false,colorlinks=false]
 {hyperref}

\makeatletter

\pdfpageheight\paperheight
\pdfpagewidth\paperwidth

\usepackage{enumitem}		
\newlength{\lyxlabelwidth}      
  \theoremstyle{plain}
  \newtheorem*{thm*}{\protect\theoremname}
\theoremstyle{plain}
\newtheorem{thm}{\protect\theoremname}
  \theoremstyle{plain}
  \newtheorem{lem}[thm]{\protect\lemmaname}
\newenvironment{elabeling}[2][]%
{\settowidth{\lyxlabelwidth}{#2}
\begin{description}[font=\normalfont,style=sameline,
leftmargin=\lyxlabelwidth,#1]}
{\end{description}}

\date{}

\makeatother

  \providecommand{\lemmaname}{Lemma}
  \providecommand{\theoremname}{Theorem}
\providecommand{\theoremname}{Theorem}

\begin{document}
\global\long\def\bR{\mathbf{R}}
\global\long\def\Bal{\mathrm{Bal}}
\global\long\def\Dirac{\mathrm{Dirac}}
\global\long\def\bC{\mathbf{C}}
\global\long\def\bD{\mathbb{D}}
\global\long\def\sign{\mathrm{sign}}
\global\long\def\coni{C_{1}}
\global\long\def\conxvi{C_{2}}
\global\long\def\conxix{C_{3}}
\global\long\def\conxviii{C_{4}}
\global\long\def\conxxv{C_{5}}
\global\long\def\conxxvi{C_{6}}
\global\long\def\conxxvii{C_{7}}

\title{Some notes on $L^{p}$ Bernstein inequality when $0<p<1$ }

\author{Béla Nagy and Tamás Varga}
\maketitle
\begin{abstract}
Recently, Nagy-Toókos and Totik-Varga proved an asymptotically sharp
$L^{p}$ Bernstein type inequality on union of finitely many intervals.
We extend this inequality to the case when the power $p$ is between
$0$ and $1$; such sharp Bernstein type inequality was proved first
by Arestov. 

Classification: 41A17, 41A44

\end{abstract}

\section{Introduction and new results}

Bernstein inequality from approximation theory is well known. In the
last decades it was generalized to arbitrary compact subsets of the
real line using potential theoretical quantities. For potential theory,
we refer to the books \cite{MR1334766} and \cite{MR1485778}. This
general form of Bernstein inequality states that if $K\subset\mathbf{R}$
is a compact set, $x\in K$ interior point, $P$ is an algebraic polynomial,
then 
\begin{equation}
\left|P'\left(x\right)\right|\le\deg\left(P\right)\pi\omega_{K}\left(x\right)\left\Vert P\right\Vert _{K}\label{ineq:bt}
\end{equation}
where $\omega_{K}\left(x\right)$ is the density of the equilibrium
measure. Inequality \eqref{ineq:bt} was proved independently by Baran
\cite{MR1160251} and Totik \cite{MR1864632}.

Recently, this inequality was generalized to $L^{p}$ norms for algebraic
polynomials and trigonometric polynomials, see \cite{nagytookos}
and \cite{totikvarga}. We need the following notation for $E\subset\left[0,2\pi\right):$
\[
\Gamma_{E}:=\left\{ e^{it}:\ t\in E\right\} .
\]
The following is Theorem 1.1 in \cite{totikvarga}.
\begin{thm*}
Let $1\le p<\infty$ and $E\subset[0,2\pi)$ be a compact set consisting
of finitely many intervals. Denote the density of the equilibrium
measure of $\Gamma_{E}$ by $\omega_{\Gamma_{E}}\left(e^{it}\right)$.
Then, for any trigonometric polynomial $T_{n}$ of degree $n$, we
have
\begin{equation}
\int_{E}\left|\frac{T_{n}'\left(t\right)}{n\,2\pi\,\omega_{\Gamma_{E}}\left(e^{it}\right)}\right|^{p}\omega_{\Gamma_{E}}\left(e^{it}\right)dt\le\left(1+o\left(1\right)\right)\int_{E}\left|T_{n}\left(t\right)\right|^{p}\omega_{\Gamma_{E}}\left(e^{it}\right)dt\label{ineq:lpbern}
\end{equation}
where $o\left(1\right)$ tends to $0$ uniformly in $T_{n}$ as $n\rightarrow\infty$.
\end{thm*}
In this paper we extend inequality \eqref{ineq:lpbern} to the case
$0<p<1$ (Arestov case) and show that the result is asymptotically
sharp. That is, we are going to prove the following two theorems.
\begin{thm}
\label{thm:arestov}Let $0<p<1$ be arbitrary. Let $E\subset[0,2\pi)$
be a compact set consisting of finitely many intervals. Then, for
any trigonometric polynomial $T_{n}$ of degree $n$, we have
\begin{equation}
\int_{E}\left|\frac{T_{n}'\left(t\right)}{n\,2\pi\,\omega_{\Gamma_{E}}\left(e^{it}\right)}\right|^{p}\omega_{\Gamma_{E}}\left(e^{it}\right)dt\le\left(1+o\left(1\right)\right)\int_{E}\left|T_{n}\left(t\right)\right|^{p}\omega_{\Gamma_{E}}\left(e^{it}\right)dt\label{ineq:lpbernare}
\end{equation}
where $o\left(1\right)$ tends to $0$ uniformly in $T_{n}$ as $n\rightarrow\infty$.
\end{thm}

\begin{thm}
\label{thm:sharpness}Let $0<p<1$ be arbitrary. Let $E\subset[0,2\pi)$
be a compact set consisting of finitely many intervals. Then, there
exist trigonometric polynomials $T_{n}$ such that
\[
\int_{E}\left|\frac{T_{n}'\left(t\right)}{n\,2\pi\,\omega_{\Gamma_{E}}\left(e^{it}\right)}\right|^{p}\omega_{\Gamma_{E}}\left(e^{it}\right)dt\ge\left(1-o\left(1\right)\right)\int_{E}\left|T_{n}\left(t\right)\right|^{p}\omega_{\Gamma_{E}}\left(e^{it}\right)dt
\]
where $o\left(1\right)$ tends to $0$ as $n\rightarrow\infty$.
\end{thm}

\section{Approach of Theorem \ref{thm:arestov} \label{sec:arestov-case}}

The proof of Theorem \ref{thm:arestov} follows the same idea as that
of \eqref{ineq:lpbern}. First, we introduce notations and cite results
from the papers \cite{nagytookos} and \cite{totikvarga}. Then we
discuss the three lemmas which have to be modified and prove them.

We say $E$ is a T-set of order $N$ if there is a real trigonometric
polynomial $U_{N}$ of degree $N$ such that $U_{N}\left(t\right)$
runs through $\left[-1,1\right]$ $2N$-times as $t$ runs through
$E$, see \cite{totikvarga}, p. 404.

The idea of the proof of \eqref{ineq:lpbern} consists of three major
steps:
\begin{description}
\item [{(a)}] To prove it when $E$ is a so-called T-set associated of
the trigonometric polynomial $U_{N}$, $E=U_{N}^{-1}\left[-1,1\right]$
and $T_{n}$ is a polynomial of $U_{N}$;
\item [{(b)}] to prove it when $E$ is a T-set and $T_{n}$ is an arbitrary
polynomial;
\item [{(c)}] to prove it when both the finite interval-system $E$ and
the trigonometric polynomial $T_{n}$ are arbitrary.
\end{description}
To verify (a), it is enough to use Arestov inequality for trigonometric
polynomials, see \cite{MR1261635}, Section 4.3 or \cite{MR607574}
instead of Zygmund inequality, all the other parts of that proof come
through.

As for part (c), the proof from \cite{totikvarga} can be applied
here since they did not use the fact that $1\le p<\infty$.

As for part (b), most of the ideas from \cite{totikvarga} (and \cite{nagytookos})
applies here as well, except for Lemma 3.3 and Proposition 3.5 from
\cite{totikvarga}. To adopt the latters to our case, we need to recall
some notations and some details of that approach.

\bigskip{}

We split the set $E$ as follows (see Section 3 in \cite{totikvarga}).
Let 
\[
E=\cup_{l=1}^{m}\left[v_{2l-1},v_{2l}\right]=\cup_{l=1}^{m}\cup_{h=1}^{r_{l}}\left[\zeta_{l,h-1},\zeta_{l,h}\right]=\cup_{j=1}^{2N}B_{j}
\]
where $\left[v_{2l-1},v_{2l}\right]$, $l=1,2,\dots,m,$ denote the
components of $E$ and $\left[\zeta_{l,h-1},\zeta_{l,h}\right]=B_{r_{1}+\dots+r_{l-1}+h}$,
$l=1,2,\dots,m$, $h=h_{l}=1,2,\dots,r_{l},$ denote the branches
of $E$. Recall that a subset of $E$ is called branch if $U_{N}$
is strictly monotone on this subset, and $U_{N}\left(t\right)$ runs
through $\left[-1,1\right]$ precisely once as $t$ runs through this
subset. If two branches, $B_{j}$ and $B_{j+1}$ has a common point,
say $\zeta$, then necessarily $\left|U_{N}\left(\zeta\right)\right|=1$
and $U_{N}'\left(\zeta\right)=0$ and we say that $\zeta$ is an inner
extremal point.

We fix $\gamma,\kappa,\text{\ensuremath{\theta}}$  such that 
\[
1/2>\theta>4\kappa
\]
and 
\[
0<\gamma<\frac{\kappa}{2}.
\]
There will be one more assumption. We divide $E$ into intervals $I_{j}$
of length between $1/2n^{\kappa}$ and $1/n^{\kappa}$ as in \cite{totikvarga},
Subsection 3.1. We call such $I_{j}$ a small interval. $J_{n}$ denotes
the set of indices of the small intervals. Let $J\subset J_{n}$ be
arbitrary. Then $H=H(J)$ denotes the union of $\left\{ I_{j}\right\} _{j\in J}$
and $H_{b}$ denotes the union of the bordering small intervals $I_{j}$,
for precise definition we refer to Subsection 3.1 in \cite{totikvarga}.
We need the following notations:
\begin{gather*}
A\left(T_{n,}X\right):=\int_{X}\left|\frac{T_{n}'(t)}{n2\pi\omega_{\Gamma_{E}}\left(e^{it}\right)}\right|^{p}\omega_{\Gamma_{E}}\left(e^{it}\right)\mathrm{d}t,\\
B\left(T_{n},X\right):=\int_{X}\left|T_{n}(t)\right|^{p}\omega_{\Gamma_{E}}\left(e^{it}\right)\mathrm{d}t,\\
a\left(T_{n},X\right):=\frac{A\left(T_{n},X\right)}{A\left(T_{n},E\right)},
\end{gather*}
and

\[
b\left(T_{n},X\right):=\frac{B\left(T_{n},X\right)}{B\left(T_{n},E\right)}.
\]
With these notations a set $H=H\left(J\right)$ can possess the following
properties:
\begin{equation}
\left(H\cup H_{b}\right)\cap E\subset\left[\zeta_{l,h},\zeta_{l,h+1}\right]\tag{I}\label{eq:I}
\end{equation}
for some $l$ and $h$;
\begin{equation}
a\left(T_{n},H_{b}\right)\le n^{-\gamma}\tag{II-a},\label{eq:II-a}
\end{equation}
\begin{equation}
b\left(T_{n},H_{b}\right)\le n^{-\gamma}\tag{II-b},\label{eq:II-b}
\end{equation}
and

\begin{equation}
\left|H(J)\right|\le4n^{\gamma-\kappa}=o(1)\tag{III}.\label{eq:III}
\end{equation}
With this splitting at hand we use another division of $E$. An interval
$H=H\left(J\right)\subset E$ can be of first, second or third type.
The union of them covers $E$ except at most $4N$ small intervals.
For the definition, we refer to \cite{totikvarga} Subsection 3.2
(Three types of subinterval; see also \cite{nagytookos} Subsection
5.2). What is important for us is that an interval $H$ of the first
or third type possesses the properties \eqref{eq:I}, \eqref{eq:II-a}
and \eqref{eq:II-b}, while an interval $H$ of the second type has
the properties \eqref{eq:II-a}, \eqref{eq:II-b} and \eqref{eq:III}
and there is exactly one $l\in\left\{ 1,2,\dots,m\right\} $ and exactly
one $h=h_{l}\in\left\{ 1,2,\dots,r_{l}-1\right\} $ such that $\zeta_{l,h_{l}}\in H,$
so an interval of the second type with its bordering small intervals
has intersections of positive Lebesgue measures with two branches
of $E$. 

First we cite Lukashov's result \cite{MR2069196}, which extends the
trigonometric form of Bernstein-Szegő's inequality from an interval
to an arbitrary compact subset of $(-\pi,\pi]$. Actually, we use
this following, special case only.
\begin{lem}
Let $E\subset\left(-\pi,\pi\right]$ consist of finitely many intervals.
If $e^{it}$ is an inner point of $\Gamma_{E}$, then for any trigonometric
polynomial $T_{n}$ of degree at most $n=1,2,\ldots$ we have
\begin{equation}
\left|T_{n}'\left(t\right)\right|\le n\,2\pi\,\omega_{\Gamma_{E}}\left(e^{it}\right)\left\Vert T_{n}\right\Vert _{E}.\label{Luk}
\end{equation}

\end{lem}
The following lemma states that a T-set $E$ of order $N$ can be
deformed such a way that it remains a T-set of order $N$, see also
Proposition 3.5 in \cite{totikvarga}.
\begin{lem}
\label{sulykonv} Let $E$ be the union of the disjoint intervals
$\left[v_{2l-1},v_{2l}\right]\subset\left(-\pi,\pi\right)$, where
$l=1,2,\dots,m$ and $v_{2l-1}<v_{2l}<v_{2l+1}$. If $E$ is a T-set
of order $N$, $E=U^{-1}\left[-1,1\right]$, then there is a family
of $E\left(\delta\right)$, $\delta>0$ of T-sets of order $N$ such
that \end{lem}
\begin{elabeling}{00.00.0000}
\item [{(i)}] 
\[
E\left(\delta\right)=\bigcup_{l=1}^{m}\left[v_{2l-1},v_{2l}\left(\delta\right)\right],
\]
 where each $v_{2l}\left(\delta\right)$ strictly decreases in $\delta$
and converges to $v_{2l}$ for every $l\in\{1,2,\dots,m\}$ and $E\left(\delta\right)\subset E$; 
\item [{(ii)}] if $U$ has the inner extremal points $\zeta_{l,1}<\zeta_{l,2}<\dots<\zeta_{l,r_{l}-1}$
in its $l$-th component $\left[v_{2l-1},v_{2l}\right]$ then $E\left(\delta\right)$
also has $r_{l}-1$ inner extremal points $\zeta_{l,1}(\delta)<\zeta_{l,2}\left(\delta\right)<\dots<\zeta_{l,r_{l}-1}\left(\delta\right)$
in $\left[v_{2l-1},v_{2l}\left(\delta\right)\right]$ such that each
$\zeta_{l,h}\left(\delta\right)$ strictly decreases in $\delta$
and converges to $\zeta_{l,h}$, $h=h_{l}\in\left\{ 1,2,\dots,r_{l}-1\right\} $; 
\item [{(iii)}] if $\omega_{\Gamma_{E}}$, $\omega_{\Gamma_{E\left(\delta\right)}}$
denote the corresponding equilibrium densities of $\Gamma_{E}$ and
$\Gamma_{E\left(\delta\right)}$ then there is a sequence $D_{\delta}=D(E\left(\delta\right))\to1$
for which the estimates 
\begin{equation}
1\le\frac{\omega_{\Gamma_{E\left(\delta\right)}}\left(e^{it}\right)}{\omega_{\Gamma_{E}}\left(e^{it}\right)}\le D_{\delta}\label{w/w}
\end{equation}
 are valid for every 
\begin{equation}
t\in\bigcup_{l=1}^{m}\left[\frac{v_{2l-1}+\zeta_{l,1}}{2},\frac{\zeta_{l,r_{l}-1}+v_{2l}}{2}\right].\label{set:core}
\end{equation}
 and for sufficiently small $\delta>0$. 
\end{elabeling}
Note that if $U$ is strictly monotone on $\left[v_{2l-1},v_{2l}\right]$
for some $l$, then $r_{l}=0$ and we let $\left[\frac{v_{2l-1}+\zeta_{l,1}}{2},\frac{\zeta_{l,r_{l}-1}+v_{2l}}{2}\right]:=\emptyset$.

We need to approximate the characteristic function of an interval
by a trigonometric polynomial of small degree, see Lemma 3.1 in \cite{totikvarga}.
\begin{lem}
\label{approx} Fix $0<p<\infty$. We know that $1/2>\theta>4\kappa$.
Then there are constants $\coni,\conxvi>0$ with the following properties.
Assume that $H=H(J)$ ($J\subset J_{n}$) is an interval with characteristic
function $\chi_{H}(t)$. There exists a trigonometric polynomial $q=q(H,n;t)$
with $\deg\left(q\right)\le2n^{2\theta}+1\le3n^{2\theta}$ which satisfies
\begin{equation}
0\le q\left(t\right)\le1,\quad t\in\left[-\pi,\pi\right),\label{qnorm1}
\end{equation}
furthermore,

\begin{equation}
\left|q(t)-\chi_{H}(t)\right|\le\conxvi e^{-\coni n^{\theta}},\label{approx1}
\end{equation}

\[
\left|q'(t)\right|\le\conxvi e^{-\coni n^{\theta}},
\]
whenever $t\in\left[-\pi,\pi\right]\setminus H_{b}.$
\end{lem}
The independence of constant $\conxvi$  of $E$ and the degree estimate
can be seen from following the proof in \cite{nagytookos}. 

Let $F_{n}:=\conxvi\exp\left(-\coni n^{\theta}\right)$. Later we
will need that there exists $\conxix>0$ such that
\begin{equation}
F_{n}^{p}\le\conxix n^{-\gamma}.\label{est:fp_becsles}
\end{equation}
Of course, $\conxix$ depends on $\gamma$, $\theta$, $\coni$ and
$\conxvi$ only.

There is one more assumption mentioned earlier. This was not present
when the power was greater than $1$ (here $p\le1$). The assumption
is 
\begin{equation}
\left(1-2\theta\right)p\ge\gamma.\tag{{IV}}\label{eq:uj_felt}
\end{equation}
For example, $\theta=1/4$, $\kappa=1/32$ and $\gamma=\min\left(1/65,p/2\right)$
is a good choice.

Next lemma says that the integral of a trigonometric polynomial cannot
be arbitrarily small, see Lemma 3.12 in \cite{totikvarga}. This is
a Nikolskii type inequality.
\begin{lem}
\label{also} Let $0<p<\infty$, $E$ consist of finitely many intervals,
$I$ be a fixed subinterval of $E$ and let $T$ be an arbitrary trigonometric
polynomial with the property $\sup_{t\in I}|T\left(t\right)|=1$.
Then, there exists $\conxviii>0$ depending on the length of $I$
only (and is independent of $E$, $p$ and $T$) such that 
\[
\int_{I}\left|T\left(t\right)\right|^{p}\omega_{\Gamma_{E}}\left(e^{it}\right)\,\mathrm{d}t\ge\conxviii\frac{1}{2^{p}}\frac{1}{\left(\deg T\right)^{2}}.
\]

\end{lem}
The following lemma is a symmetrization technique for trigonometric
polynomials. For proof and details, we refer to \cite{totikvarga},
see Lemma 3.2.
\begin{lem}
Let $E$ be a T-set associated with the trigonometric polynomial $U_{N}$
of degree $N$. For a point $t\in E$ with $U_{N}(t)\in(-1,1)$ let
$t_{1},t_{2},\dots,t_{2N}$ be those points in $E$ which satisfy
$U_{N}\left(t_{h}\right)=U_{N}(t)$. If $V_{n}$ is a trigonometric
polynomial of degree at most $n$, then there is an algebraic polynomial
$S_{n/N}$ of degree at most $n/N$ such that 
\[
\sum_{h=1}^{2N}V_{n}\left(t_{h}\right)=S_{[n/N]}\left(U_{N}\left(t\right)\right).
\]

\end{lem}
If $U_{N,h}^{-1}$ denotes the inverse of $U_{N}$ restricted on the
branch $B_{h}$ then $t_{h}=U_{N,h}^{-1}\left(U_{N}\left(t\right)\right).$
This shows that 
\begin{equation}
\frac{\mathrm{d}}{\mathrm{d}t}t_{h}=\frac{U_{N}'(t)}{U_{N}'\left(t_{h}\right)}.\label{t_ideriv}
\end{equation}
It is known (see, e.g., Lemma 3.1 in \cite{MR2961187}) that
\begin{equation}
\omega_{\Gamma_{E}}\left(e^{it}\right)=\frac{1}{2\pi N}\frac{\left|U_{N}'(t)\right|}{\sqrt{1-U_{N}\left(t\right)^{2}}}.\label{eq_meas}
\end{equation}
The previous equation, together with \eqref{t_ideriv}, shows that
\begin{equation}
\left|\frac{\omega_{\Gamma}\left(e^{it_{h}}\right)\frac{\mathrm{d}}{\mathrm{d}t}t_{h}}{\omega_{\Gamma}\left(e^{it}\right)}\right|=1.\label{egyszerusites}
\end{equation}

Let $T_{n}$ be a trigonometric polynomial of degree $n,$ and let
$q$ be the trigonometric polynomial from Lemma \ref{approx}. Then,
by the previous lemma, 
\[
T_{n}^{*}\left(t\right):=\sum_{h=1}^{2N}T_{n}\left(t_{h}\right)q\left(t_{h}\right)
\]
is a polynomial of the trigonometric polynomial $U_{N}$, so $T_{n}^{*}$
makes a link between step (a) and step (b). This connection is shown
by the the analogue Lemmas \ref{lem:7} and \ref{lem:8} of Lemma
3.3 from \cite{totikvarga} (see also Lemmas 7, 8 in \cite{nagytookos}).
We will prove the following three lemmas in the next subsection.
\begin{lem}
\label{lem:7}Let $0<p<1$. Suppose we have a $T$-set $E$ associated
with the trigonometric polynomial $U_{N}$ of degree $N$. We also
have an interval $H=H\left(J\right)$ satisfying the property \eqref{eq:I}.
Then, using $T_{n}^{*}$ defined for $T_{n}$ and $n^{*}:=\deg\left(T_{n}^{*}\right)\le\deg\left(T_{n}\right)+\deg\left(q\right)$,
we have
\begin{multline*}
\left|\left(\frac{n^{*}}{n}\right)^{p}A\left(T_{n}^{*},E\right)-2N\, A\left(T_{n},H\right)\right|\\
\le2N\left(4F_{n}^{p}+a\left(T_{n},H_{b}\right)\right)\, A\left(T_{n},E\right)+\left(2N\right)4\cdot3^{p}\left(\frac{n^{2\theta}}{n}\right)^{p}\, B\left(T_{n},E\right).
\end{multline*}

\end{lem}

\begin{lem}
\label{lem:8}With the same assumption as in Lemma \ref{lem:7} we
also have that

\[
\left|B\left(T_{n}^{*},E\right)-2N\, B\left(T_{n},H\right)\right|\le2N\left(3F_{n}^{p}+b\left(T_{n},H_{b}\right)\right)B\left(T_{n},E\right).
\]

\end{lem}
Analogously to $A\left(T_{n},X\right)$ and $B\left(T_{n},X\right)$
we use the notations $A_{\delta}\left(T_{n},X\right)$ and $B_{\delta}\left(T_{n},X\right)$
:

\[
A_{\delta}\left(T_{n},X\right)=\int_{X}\left|\frac{T_{n}'(t)}{n\,2\pi\,\omega_{\Gamma_{E\left(\delta\right)}}\left(e^{it}\right)}\right|^{p}\omega_{\Gamma_{E\left(\delta\right)}}\left(e^{it}\right)\,\mathrm{d}t,
\]
and
\[
B_{\delta}\left(T_{n},X\right)=\int_{X}\left|T_{n}(t)\right|^{p}\omega_{\Gamma_{E\left(\delta\right)}}\left(e^{it}\right)\,\mathrm{d}t
\]
respectively, where $E\left(\delta\right)$ comes from Lemma \ref{sulykonv}.

Next lemma corresponds to Proposition 3.5 from \cite{totikvarga}.
\begin{lem}
\label{lem:11} Let $0<p<1$. Assume that $H$ is an interval of second
type. Let $q=q(H,n;t)$ be the polynomial from Lemma \ref{approx}
and let $X$ be an arbitrary subset of $E.$ Then the following  estimates
hold:
\begin{multline}
\left|A\left(T_{n}q,H\right)-A\left(T_{n},H\right)\right|\\
\le\left(F_{n}^{p}+3^{p}\left(\frac{n^{2\theta}}{n}\right)^{p}\right)A\left(T_{n},E\right)+3^{p}\left(\frac{n^{2\theta}}{n}\right)^{p}B\left(T_{n},E\right),\label{eq:3}
\end{multline}
\begin{equation}
A\left(T_{n}q,X\right)\le A\left(T_{n},X\right)+3^{p}\left(\frac{n^{2\theta}}{n}\right)^{p}B\left(T_{n},E\right),\label{eq:3b}
\end{equation}
\begin{equation}
B\left(T_{n}q,X\right)\le B\left(T_{n},X\right).\label{eq:4}
\end{equation}
We also have, if $n$ is large, 
\begin{multline}
\left|A_{\delta}\left(T_{n}q,E\left(\delta\right)\right)-A_{\delta}\left(T_{n}q,H\right)\right|\le D_{\delta}^{1-p}a\left(T_{n},H_{b}\right)A\left(T_{n},E\right)\\
+\left(D_{\delta}^{1-p}F_{n}^{p/2}+D_{\delta}^{1-p}3^{p}\left(\frac{n^{2\theta}}{n}\right)^{p}+\left(\frac{n^{2\theta}}{n}\right)^{p}b\left(T_{n},H_{b}\right)\right)B\left(T_{n},E\right),\label{eq:5}
\end{multline}
and
\begin{equation}
\left|B_{\delta}\left(T_{n}q,E\left(\delta\right)\right)-B_{\delta}\left(T_{n}q,H\right)\right|\le D_{\delta}\left(F_{n}^{p}+b\left(T_{n},H_{b}\right)\right)B\left(T_{n},E\right).\label{eq:6}
\end{equation}

\end{lem}

\subsection{Proofs of the Lemmas \ref{lem:7}, \ref{lem:8} and \ref{lem:11} }

\subsubsection{Proofs of the Lemmas }

Interestingly, the proofs of Lemmas \ref{lem:7} and \ref{lem:8}
are somewhat simpler than those of in Nagy-Toókos \cite{nagytookos}
and Totik-Varga \cite{totikvarga}. This ,,simplicity'' derives from
the following inequalities: 
\begin{equation}
\bigl||a|^{p}-|b|^{p}\bigr|\le|a-b|^{p}\label{ineq:1}
\end{equation}
\begin{equation}
|a+b|^{p}\le|a|^{p}+|b|^{p},\label{ineq:2}
\end{equation}
where $a,b\in\mathbb{R}$ and $0<p<1$.

During the verification of Lemmas \ref{lem:7} and \ref{lem:8} we
frequently use the subsequent identities (cf. formulas (61), (62)
in \cite{nagytookos}). Let $X$ be a subset of the branch $B_{h_{0}}$
and let $X_{h}$ denote the set $U_{N,h}^{-1}\bigl(U_{N}(X)\bigr)$.
Then

\begin{equation}
\int_{X}f\left(t_{h}\right)\omega_{\Gamma_{E}}\left(e^{it}\right)\,\mathrm{d}t=\int_{X_{h}}f\left(t\right)\omega_{\Gamma_{E}}\left(e^{it}\right)\,\mathrm{d}t,\label{helyettesit1}
\end{equation}
and
\begin{equation}
\int_{X}\left(\frac{\left(f\left(t_{h}\right)\right)'}{\omega_{\Gamma_{E}}\left(e^{it}\right)}\right)^{p}\omega_{\Gamma_{E}}\left(e^{it}\right)\,\mathrm{d}t=\int_{X_{h}}\left(\frac{\left(f\left(t\right)\right)'}{\omega_{\Gamma_{E}}\left(e^{it}\right)}\right)^{p}\omega_{\Gamma_{E}}\left(e^{it}\right)\,\mathrm{d}t\label{helyettesit2}
\end{equation}
respectively.
\begin{proof}
[Proof of Lemma \ref{lem:7}] 

\begin{multline}
\left|\left(\frac{n^{*}}{n}\right)^{p}A(T_{n}^{*},E)-2NA(T_{n},H)\right|\\
\le\left|\left(\frac{n^{*}}{n}\right)^{p}A\Bigl(T_{n}^{*},E\setminus U_{N}^{-1}\bigl(U_{N}(H\cup H_{b})\bigr)\Bigr)\right|+\left|\left(\frac{n^{*}}{n}\right)^{p}A\Bigl(T_{n}^{*},U_{N}^{-1}\bigl(U_{N}(H_{b})\bigr)\Bigr)\right|\\
+\left|\left(\frac{n^{*}}{n}\right)^{p}A\Bigl(T_{n}^{*},U_{N}^{-1}\bigl(U_{N}(H)\bigr)\Bigr)-2NA(T_{n},H)\right|\label{lem:7Min}
\end{multline}
We separately estimate the terms of the right hand side. We begin
with the first term. By \eqref{ineq:2}, we get

\begin{multline}
\left(\frac{n^{*}}{n}\right)^{p}A\Bigl(T_{n}^{*},E\setminus U_{N}^{-1}\bigl(U_{N}(H\cup H_{b})\bigr)\Bigr)\\
\le\sum_{h=1}^{2N}\int_{E\setminus U_{N}^{-1}\bigl(U_{N}(H\cup H_{b})\bigr)}\left|\frac{\bigl(T_{n}(t_{h})\bigr)'q(t_{h})}{n2\pi\omega_{\Gamma_{E}}(e^{it})}\right|^{p}\omega_{\Gamma_{E}}(e^{it})\,\mathrm{d}t\\
+\sum_{h=1}^{2N}\int_{E\setminus U_{N}^{-1}\bigl(U_{N}(H\cup H_{b})\bigr)}\left|\frac{T_{n}(t_{h})\bigl(q(t_{h})\bigr)'}{n2\pi\omega_{\Gamma_{E}}(e^{it})}\right|^{p}\omega_{\Gamma_{E}}(e^{it})\,\mathrm{d}t.\label{eq:a_tavol_elso_lepes}
\end{multline}
On $E\setminus U_{N}^{-1}\bigl(U_{N}(H\cup H_{b})\bigr)$ we know
that $t_{h}\notin H\cup H_{b}$, therefore by \eqref{approx1}, $\left|q\left(t_{h}\right)\right|\le\conxvi e^{-\coni n^{\theta}}=F_{n}.$
This fact is applied to the first term of the right hand side of \eqref{eq:a_tavol_elso_lepes}.
As regards the second term, we use Lukashov's inequality \eqref{Luk}
and \eqref{egyszerusites}. Then increasing the domain of integration
$E\setminus U_{N}^{-1}\bigl(U_{N}(H\cup H_{b})\bigr)$ to $E=\cup_{j=1}^{2N}B_{j}$
the previous inequality is continued as

\begin{multline*}
\le F_{n}^{p}\sum_{h=1}^{2N}\sum_{j=1}^{2N}\int_{B_{j}}\left|\frac{\left(T_{n}\left(t_{h}\right)\right)'}{n2\pi\omega_{\Gamma_{E}}\left(e^{it}\right)}\right|^{p}\omega_{\Gamma_{E}}\left(e^{it}\right)\,\mathrm{d}t\\
+\left(\frac{\deg(q)}{n}\right)^{p}\sum_{h=1}^{2N}\sum_{j=1}^{2N}\int_{B_{j}}\left|T_{n}\left(t_{h}\right)\right|^{p}\omega_{\Gamma_{E}}\left(e^{it}\right)\,\mathrm{d}t.
\end{multline*}
By \eqref{helyettesit1} and \eqref{helyettesit2} we get

\begin{multline*}
\le2NF_{n}^{p}\sum_{h=1}^{2N}\int_{B_{h}}\left|\frac{\left(T_{n}\left(t\right)\right)'}{n2\pi\omega_{\Gamma_{E}}\left(e^{it}\right)}\right|^{p}\omega_{\Gamma_{E}}\left(e^{it}\right)\,\mathrm{d}t\\
+2N\left(\frac{\deg(q)}{n}\right)^{p}\sum_{h=1}^{2N}\int_{B_{h}}\left|T_{n}(t)\right|^{p}\omega_{\Gamma_{E}}\left(e^{it}\right)\,\mathrm{d}t\\
=o(1)A(T_{n},E)+o(1)B(T_{n},E),
\end{multline*}
remembering the facts that $F_{n}=o(1)$ and $\deg\left(q\right)/n=o\left(1\right)$.
That is we have
\begin{multline}
\left(\frac{n^{*}}{n}\right)^{p}A\left(T_{n}^{*},E\setminus U_{N}^{-1}\left[U_{N}\left(H\cup H_{b}\right)\right]\right)\\
\le\left(2N\right)F_{n}^{p}A\left(T_{n},E\right)+\left(2N\right)\left(\frac{3n^{2\theta}}{n}\right)^{p}B\left(T_{n},E\right)\label{eq:lem8_est1}
\end{multline}
or, without detailing the error terms, we can write

\begin{equation}
\left(\frac{n^{*}}{n}\right)^{p}A\left(T_{n}^{*},E\setminus U_{N}^{-1}\left[U_{N}\left(H\cup H_{b}\right)\right]\right)\le o\left(1\right)A\left(T_{n},E\right)+o\left(1\right)B\left(T_{n},E\right).\label{first1}
\end{equation}

Now we turn to the second term of the right hand side of \eqref{lem:7Min}.
Denote by $h_{0}$ the index of the branch which $(H\cup H_{b})\cap E$
belongs to (such index exists because of the property \eqref{eq:I}).
By \eqref{ineq:2} we obtain
\begin{multline}
\left(\frac{n^{*}}{n}\right)^{p}A\left(T_{n}^{*},U_{N}^{-1}\left[U_{N}\left(H_{b}\right)\right]\right)\\
\le\int_{U_{N}^{-1}\bigl(U_{N}(H_{b})\bigr)}\left|\frac{\bigl(T_{n}(t_{h_{0}})\bigr)'q(t_{h_{0}})}{n2\pi\omega_{\Gamma_{E}}(e^{it})}\right|^{p}\omega_{\Gamma_{E}}(e^{it})\,\mathrm{d}t\\
+\sum_{h\not=h_{0}}\int_{U_{N}^{-1}\bigl(U_{N}(H_{b})\bigr)}\left|\frac{\bigl(T_{n}(t_{h})\bigr)'q(t_{h})}{n2\pi\omega_{\Gamma_{E}}(e^{it})}\right|^{p}\omega_{\Gamma_{E}}(e^{it})\,\mathrm{d}t\\
+\int_{U_{N}^{-1}\bigl(U_{N}(H_{b})\bigr)}\left|\frac{T_{n}(t_{h_{0}})\bigl(q(t_{h_{0}})\bigr)'}{n2\pi\omega_{\Gamma_{E}}(e^{it})}\right|^{p}\omega_{\Gamma_{E}}(e^{it})\,\mathrm{d}t\\
+\sum_{h\ne h_{0}}\int_{U_{N}^{-1}\bigl(U_{N}(H_{b})\bigr)}\left|\frac{T_{n}(t_{h})\bigl(q(t_{h})\bigr)'}{n2\pi\omega_{\Gamma_{E}}(e^{it})}\right|^{p}\omega_{\Gamma_{E}}(e^{it})\,\mathrm{d}t.\label{second1}
\end{multline}
The estimations of the second and the fourth term of the right hand
side of \eqref{second1} follow that of the first term at the right
hand side of \eqref{lem:7Min}. If $h\not=h_{0}$ then \eqref{approx1}
again implies that $\left|q\left(t_{h}\right)\right|\le F_{n}$ on
the set $U_{N}^{-1}\left[U_{N}\left(H_{b}\right)\right]$, therefore,
for these two terms, the subsequent inequality holds:
\begin{multline}
\sum_{h\not=h_{0}}\int_{U_{N}^{-1}\bigl(U_{N}(H_{b})\bigr)}\left|\frac{\bigl(T_{n}(t_{h})\bigr)'q(t_{h})}{n2\pi\omega_{\Gamma_{E}}(e^{it})}\right|^{p}\omega_{\Gamma_{E}}(e^{it})\,\mathrm{d}t\\
+\sum_{h\not=h_{0}}\int_{U_{N}^{-1}\bigl(U_{N}(H_{b})\bigr)}\left|\frac{T_{n}(t_{h})\bigl(q(t_{h})\bigr)'}{n2\pi\omega_{\Gamma_{E}}(e^{it})}\right|^{p}\omega_{\Gamma_{E}}(e^{it})\,\mathrm{d}t\\
\le2NF_{n}^{p}\sum_{h\not\ne h_{0}}\int_{B_{h}}\left|\frac{\bigl(T_{n}(t)\bigr)'}{n2\pi\omega_{\Gamma_{E}}(e^{it})}\right|^{p}\omega_{\Gamma_{E}}(e^{it})\,\mathrm{d}t\\
+2N\left(\frac{\deg(q)}{n}\right)^{p}\sum_{h\not=h_{0}}\int_{B_{h}}\left|T_{n}(t)\right|^{p}\omega_{\Gamma_{E}}(e^{it})\,\mathrm{d}t\\
\le o(1)A(T_{n},E)+o(1)B(T_{n},E).\label{second1.1}
\end{multline}
For the first term of the right hand side of \eqref{second1} we get,
by $\left\Vert q\right\Vert _{E}\le1$ and by \eqref{helyettesit2},
that
\begin{multline*}
\int_{U_{N}^{-1}\bigl(U_{N}(H_{b})\bigr)}\left|\frac{\left(T_{n}\left(t_{h_{0}}\right)\right)'q\left(t_{h_{0}}\right)}{n2\pi\omega_{\Gamma_{E}}(e^{it})}\right|^{p}\omega_{\Gamma_{E}}(e^{it})\,\mathrm{d}t\\
\le\sum_{j=1}^{2N}\int_{\left((H_{b})_{j}\right)_{h_{0}}}\left|\frac{T_{n}'(t)}{n2\pi\omega_{\Gamma_{E}}(e^{it})}\right|^{p}\omega_{\Gamma_{E}}(e^{it})\,\mathrm{d}t.
\end{multline*}
Since $(H_{b})_{h_{0}}=H_{b}$ (where, by notational conventions,
$\left(H_{b}\right)_{h_{0}}=U_{N,h_{0}}^{-1}\left[U_{N}\left(H_{b}\right)\right]$),
and $\left((H_{b})_{j}\right)_{h_{0}}=\left(H_{b}\right)_{h_{0}}=H_{b}$
, we can continue as
\begin{equation}
=2NA(T_{n},H_{b})=2Na(T_{n},H_{b})A(T_{n},E).\label{second1.2}
\end{equation}
In order to estimate the third term of the right hand side of \eqref{second1}
we use Lukashov's inequality \eqref{Luk}, \eqref{helyettesit1} and
\eqref{egyszerusites}. Then 
\begin{multline*}
\int_{U_{N}^{-1}\bigl(U_{N}(H_{b})\bigr)}\left|\frac{T_{n}\left(t_{h_{0}}\right)\left(q\left(t_{h_{0}}\right)\right)'}{n2\pi\omega_{\Gamma_{E}}(e^{it})}\right|^{p}\omega_{\Gamma_{E}}(e^{it})\,\mathrm{d}t\\
\le\sum_{j=1}^{2N}\int_{(H_{b})_{j}}\left|\frac{T_{n}(t_{h_{0}})\deg(q)2\pi\omega_{\Gamma_{E}}(e^{it_{h_{0}}})\frac{\mathrm{d}}{\mathrm{d}t}t_{h_{0}}}{n2\pi\omega_{\Gamma_{E}}(e^{it})}\right|^{p}\omega_{\Gamma_{E}}(e^{it})\,\mathrm{d}t\\
=2N\left(\frac{\deg(q)}{n}\right)^{p}\int_{H_{b}}\left|T_{n}(t)\right|^{p}\omega_{\Gamma_{E}}(e^{it})\,\mathrm{d}t\le o\left(1\right)B\left(T_{n},E\right).
\end{multline*}
This, together with \eqref{second1}, \eqref{second1.1} and \eqref{second1.2},
implies that 
\begin{multline}
\left(\frac{n^{*}}{n}\right)^{p}A\left(T_{n}^{*},U_{N}^{-1}\left[U_{N}\left(H_{b}\right)\right]\right)\\
\le\left(\left(2N\right)F_{n}^{p}+2Na\left(T_{n},H_{b}\right)A\left(T_{n},E\right)\right)+2\left(2N\right)\left(\frac{3n^{2\theta}}{n}\right)^{p}B\left(T_{n},E\right)\label{eq:lem8_est2}
\end{multline}
or, without detailing the error terms, we can write

\begin{equation}
\left(\frac{n^{*}}{n}\right)^{p}A\left(T_{n}^{*},U_{N}^{-1}\left[U_{N}\left(H_{b}\right)\right]\right)\le\left(o\left(1\right)+2Na\left(T_{n},H_{b}\right)\right)A\left(T_{n},E\right)+o\left(1\right)B\left(T_{n},E\right),\label{second2}
\end{equation}
where $o(1)$ is independent of $T_{n}$.

As regards the third term of the right hand side of the inequality
\eqref{lem:7Min}, consider, by \eqref{helyettesit2}, that

\[
2NA(T_{n},H)=\int_{U_{N}^{-1}\bigl(U_{N}(H)\bigr)}\left|\frac{T_{n}'(t_{h_{0}})}{n2\pi\omega_{\Gamma_{E}}(e^{it})}\right|^{p}\omega_{\Gamma_{E}}(e^{it})\,\mathrm{d}t,
\]
so, by \eqref{ineq:1}, then by \eqref{ineq:2}, we obtain

\begin{multline}
\left|\left(\frac{n^{*}}{n}\right)^{p}A\Bigl(T_{n}^{*},U_{N}^{-1}\bigl(U_{N}(H)\bigr)\Bigr)-2NA(T_{n},H)\right|\\
\le\int_{U_{N}^{-1}\bigl(U_{N}(H)\bigr)}\left|\frac{\left(T_{n}^{*}\left(t\right)\right)'-\left(T_{n}\left(t_{h_{0}}\right)\right)'}{n2\pi\omega_{\Gamma_{E}}(e^{it})}\right|^{p}\omega_{\Gamma_{E}}(e^{it})\,\mathrm{d}t\\
\le\sum_{h\not=h_{0}}\int_{U_{N}^{-1}\bigl(U_{N}(H)\bigr)}\left|\frac{\left(T_{n}\left(t_{h}\right)\right)'q\left(t_{h}\right)}{n2\pi\omega_{\Gamma_{E}}(e^{it})}\right|^{p}\omega_{\Gamma_{E}}(e^{it})\,\mathrm{d}t\\
+\sum_{h=1}^{2N}\int_{U_{N}^{-1}\bigl(U_{N}(H)\bigr)}\left|\frac{T_{n}\left(t_{h}\right)\left(q\left(t_{h}\right)\right)'}{n2\pi\omega_{\Gamma_{E}}(e^{it})}\right|^{p}\omega_{\Gamma_{E}}(e^{it})\,\mathrm{d}t\\
+\int_{U_{N}^{-1}\bigl(U_{N}(H)\bigr)}\left|\frac{\left(T_{n}\left(t_{h_{0}}\right)\right)'q\left(t_{h_{0}}\right)-\left(T_{n}\left(t_{h_{0}}\right)\right)'}{n2\pi\omega_{\Gamma_{E}}(e^{it})}\right|^{p}\omega_{\Gamma_{E}}(e^{it})\,\mathrm{d}t.\label{third1}
\end{multline}
For the first term of the right hand side we again use that, by \eqref{approx1},
$\left|q(t_{h})\right|\le F_{n}$ on the set $U_{N}^{-1}\left[U_{N}\left(H\right)\right]$
for $h\ne h_{0}.$ Hence, considering \eqref{helyettesit2}, we obtain
\begin{multline}
\sum_{h\not=h_{0}}\int_{U_{N}^{-1}\bigl(U_{N}(H)\bigr)}\left|\frac{\left(T_{n}\left(t_{h}\right)\right)'q\left(t_{h}\right)}{n2\pi\omega_{\Gamma_{E}}(e^{it})}\right|^{p}\omega_{\Gamma_{E}}(e^{it})\,\mathrm{d}t\\
\le F_{n}^{p}\sum_{h\ne h_{0}}\sum_{j=1}^{2N}\int_{\left(\left(H\right)_{j}\right)_{h}}\left|\frac{T_{n}'(t)}{n2\pi\omega_{\Gamma_{E}}(e^{it})}\right|^{p}\omega_{\Gamma_{E}}(e^{it})\,\mathrm{d}t\\
=F_{n}^{p}2N\int_{\cup_{h\ne h_{0}}\left(H\right)_{h}}\left|\frac{T_{n}'(t)}{n2\pi\omega_{\Gamma_{E}}(e^{it})}\right|^{p}\omega_{\Gamma_{E}}\left(e^{it}\right)\,\mathrm{d}t\\
\le\left(2N\right)F_{n}^{p}A\left(T_{n},E\right)=o\left(1\right)A\left(T_{n},E\right)\label{third1.1}
\end{multline}
where $\left(\left(H\right)_{j}\right)_{h}=U_{N,h}^{-1}\left[U\left(U_{N,j}^{-1}\left[U_{N}\left(H\right)\right]\right)\right]=\left(H\right)_{h}$.
The second term of the right hand side of \eqref{third1} is estimated
by the help of Lukashov's inequality \eqref{Luk}. We get
\begin{multline*}
\sum_{h=1}^{2N}\int_{U_{N}^{-1}\bigl(U_{N}(H)\bigr)}\left|\frac{T_{n}\left(t_{h}\right)\left(q\left(t_{h}\right)\right)'}{n2\pi\omega_{\Gamma_{E}}(e^{it})}\right|^{p}\omega_{\Gamma_{E}}(e^{it})\,\mathrm{d}t\\
\le\sum_{h=1}^{2N}\int_{U_{N}^{-1}\bigl(U_{N}(H)\bigr)}\left|\frac{T_{n}\left(t_{h}\right)\deg\left(q\right)2\pi q\left(t_{h}\right)\omega_{\Gamma_{E}}\left(e^{it_{h}}\right)\frac{\mathrm{d}}{\mathrm{d}t}t_{h}}{n2\pi\omega_{\Gamma_{E}}(e^{it})}\right|^{p}\omega_{\Gamma_{E}}(e^{it})\,\mathrm{d}t
\end{multline*}
This inequality is continued by \eqref{egyszerusites} and \eqref{helyettesit1}
as

\begin{multline}
\le\left(\frac{\deg\left(q\right)}{n}\right)^{p}\sum_{h=1}^{2N}\sum_{j=1}^{2N}\int_{\left(\left(H\right)_{j}\right)_{h}}\left|T_{n}(t)\right|^{p}\omega_{\Gamma_{E}}(e^{it})\,\mathrm{d}t\\
=\left(\frac{\deg\left(q\right)}{n}\right)^{p}2N\int_{\cup_{h=1}^{2N}\left(H\right)_{h}}\left|T_{n}\left(t\right)\right|^{p}\omega_{\Gamma_{E}}\left(e^{it}\right)\,\mathrm{d}t\\
\le\left(\frac{3n^{2\theta}}{n}\right)^{p}\left(2N\right)B\left(T_{n},E\right)=o\left(1\right)B\left(T_{n},E\right).\label{third1.2}
\end{multline}
By \eqref{approx1} we have on the set $U_{N}^{-1}\left[U_{N}\left(H\right)\right]$
that $\left|q\left(t_{h_{0}}\right)-1\right|\le F_{n}$. Hence the
third term of the right hand side of the inequality \eqref{third1}
can be estimated as
\begin{multline*}
\int_{U_{N}^{-1}\bigl(U_{N}(H)\bigr)}\left|\frac{\left(T_{n}\left(t_{h_{0}}\right)\right)'q\left(t_{h_{0}}\right)-\left(T_{n}\left(t_{h_{0}}\right)\right)'}{n2\pi\omega_{\Gamma_{E}}(e^{it})}\right|^{p}\omega_{\Gamma_{E}}(e^{it})\,\mathrm{d}t\\
=\int_{U_{N}^{-1}\bigl(U_{N}(H)\bigr)}\left|\frac{\left(T_{n}\left(t_{h_{0}}\right)\right)'}{n2\pi\omega_{\Gamma_{E}}(e^{it})}\right|^{p}\left|q\left(t_{h_{0}}\right)-1\right|^{p}\omega_{\Gamma_{E}}(e^{it})\,\mathrm{d}t\\
\le F_{n}^{p}\int_{U_{N}^{-1}\bigl(U_{N}(H)\bigr)}\left|\frac{\left(T_{n}\left(t_{h_{0}}\right)\right)'}{n2\pi\omega_{\Gamma_{E}}(e^{it})}\right|^{p}\omega_{\Gamma_{E}}(e^{it})\,\mathrm{d}t.
\end{multline*}
By \eqref{helyettesit2} we continue the inequality as

\begin{equation}
=F_{n}^{p}2N\int_{H}\left|\frac{T_{n}'(t)}{n2\pi\omega_{\Gamma_{E}}(e^{it})}\right|^{p}\omega_{\Gamma_{E}}(e^{it})\,\mathrm{d}t\le o(1)A(T_{n},E).\label{third1.3}
\end{equation}
So by \eqref{third1}, \eqref{third1.1}, \eqref{third1.2} and \eqref{third1.3}
we get that
\begin{multline}
\left|\left(\frac{n^{*}}{n}\right)^{p}A\left(T_{n}^{*},U_{N}^{-1}\left[U_{N}\left(H\right)\right]\right)-2NA\left(T_{n},H\right)\right|\\
\le\left(2F_{n}^{p}2N\right)A\left(T_{n},E\right)+\left(\frac{3n^{2\theta}}{n}\right)^{p}\left(2N\right)B\left(T_{n},E\right).\label{third2}
\end{multline}
Now combining \eqref{lem:7Min}, \eqref{eq:lem8_est1}, \eqref{eq:lem8_est2}
and \eqref{third2}, we can write
\begin{multline*}
\left|\left(\frac{n^{*}}{n}\right)^{p}A\left(T_{n}^{*},E\right)-2NA\left(T_{n},H\right)\right|\\
\le\left(2N\right)\left(4F_{n}^{p}+a\left(T_{n},H_{b}\right)\right)A\left(T_{n},E\right)\\
+4\left(2N\right)\left(\frac{3n^{2\theta}}{n}\right)^{p}B\left(T_{n},E\right).
\end{multline*}
This way we proved the lemma. 
\end{proof}

\begin{proof}
[Proof of Lemma \ref{lem:8}] The proof is sketched briefly, because
it is very similar to that of Lemma \ref{lem:7}. We again split $E$
into three sets: $E\setminus U_{N}^{-1}\left[U_{N}\left(H\cup H_{b}\right)\right],$
$U_{N}^{-1}\left[U_{N}\left(H_{b}\right)\right]$ and $U_{N}^{-1}\left[U_{N}\left(H\right)\right]$
and start with the inequality 
\begin{multline}
\left|B\left(T_{n}^{*},E\right)-(2N)B\left(T_{n},H\right)\right|\\
\le\left|B\left(T_{n}^{*},E\setminus U_{N}^{-1}\left[U_{N}\left(H\cup H_{b}\right)\right]\right)\right|\\
+\left|B\left(T_{n}^{*},U_{N}^{-1}\left[U_{N}\left(H_{b}\right)\right]\right)\right|\\
+\left|B\left(T_{n}^{*},U_{N}^{-1}\left[U_{N}\left(H\right)\right]\right)-\left(2N\right)B\left(T_{n},H\right)\right|.\label{lem:8Min}
\end{multline}
Now, as before, we separately estimate the three terms at the right
hand side of the inequality.

In the case of the first term we first use \eqref{approx1} and \eqref{helyettesit1}
then $E\setminus U_{N}^{-1}\left[U_{N}\left(H\cup H_{b}\right)\right]$
is increased to $E$. We get that 
\begin{multline}
\left|B\left(T_{n}^{*},E\setminus U_{N}^{-1}\left[U_{N}\left(H\cup H_{b}\right)\right]\right)\right|\le F_{n}^{p}2N\int_{\cup_{h=1}^{2N}B_{h}}\left|T_{n}(t)\right|^{p}\omega_{\Gamma_{E}}(e^{it})\,\mathrm{d}t\\
\le o(1)B(T_{n},E).\label{lem:8first}
\end{multline}

In the case of the second and the third terms we again deal with the
$h_{0}$-th term of $T_{n}^{*}$ separately and apply \eqref{helyettesit1}
and \eqref{approx1}. Then, for the second term, we have
\begin{multline}
B\left(T_{n}^{*},U_{N}^{-1}\left[U_{N}\left(H_{b}\right)\right]\right)\\
\le F_{n}^{p}\sum_{h\ne h_{0}}\sum_{j=1}^{2N}\int_{\left(\left(H_{b}\right)_{j}\right)_{h}}\left|T_{n}(t)\right|^{p}\omega_{\Gamma_{E}}(e^{it})\,\mathrm{d}t\\
+2N\int_{H_{b}}\left|T_{n}(t)\right|^{p}\omega_{\Gamma_{E}}(e^{it})\,\mathrm{d}t\\
\le F_{n}^{p}2NB(T_{n},E)+2Nb(T_{n},H_{b})B(T_{n},E).\label{lem:8second}
\end{multline}
As regards the third term, we estimate similarly as in \eqref{third1}
and we use \eqref{ineq:1} and \eqref{ineq:2} here. This way we obtain
\begin{multline}
\left|B\left(T_{n}^{*},U_{N}^{-1}\left[U_{N}\left(H\right)\right]\right)-2NB\left(T_{n},H\right)\right|\\
\le F_{n}^{p}2N\sum_{h\ne h_{0}}\int_{\left(\left(H\right)_{j}\right)_{h}}\left|T_{n}(t)\right|^{p}\omega_{\Gamma_{E}}(e^{it})\,\mathrm{d}t\\
+2N\int_{H}\left|T_{n}(t)\right|^{p}\left|q\left(t\right)-1\right|^{p}\omega_{\Gamma_{E}}(e^{it})\,\mathrm{d}t\\
\le F_{n}^{p}2N\sum_{h=1}^{2N}\int_{\left(H\right)_{h}}\left|T_{n}\left(t\right)\right|^{p}\omega_{\Gamma_{E}}(e^{it})\,\mathrm{d}t\le F_{n}^{p}2NB(T_{n},E).\label{lem8:third}
\end{multline}
Using the inequalities \eqref{lem:8Min}, \eqref{lem:8first}, \eqref{lem:8second}
and \eqref{lem8:third}, we have
\[
\left|B\left(T_{n}^{*},E\right)-\left(2N\right)B\left(T_{n},H\right)\right|\le2N\left(3F_{n}^{p}+b\left(T_{n},H_{b}\right)\right)B\left(T_{n},E\right).
\]
This way we proved the lemma.
\end{proof}

\begin{proof}
[Proof of Lemma \ref{lem:11}] We begin with the verification of \eqref{eq:3}.
\eqref{ineq:1} and \eqref{ineq:2} imply that
\begin{multline*}
\left|\int_{H}\left|\frac{\left(T_{n}\left(t\right)q\left(t\right)\right)'}{\deg\left(T_{n}q\right)2\pi\omega_{\Gamma_{E}}(e^{it})}\right|^{p}\omega_{\Gamma_{E}}(e^{it})\, dt-\int_{H}\left|\frac{T_{n}'\left(t\right)}{n2\pi\omega_{\Gamma_{E}}(e^{it})}\right|^{p}\omega_{\Gamma_{E}}(e^{it})\,\mathrm{d\mathit{t}}\right|\\
\le\int_{H}\left|\frac{\frac{n}{\deg(T_{n}q)}T_{n}'(t)q(t)-T_{n}'(t)}{n2\pi\omega_{\Gamma_{E}}(e^{it})}\right|^{p}\omega_{\Gamma_{E}}(e^{it})\,\mathrm{d}t\\
+\int_{H}\left|\frac{T_{n}(t)q'(t)}{\deg(T_{n}q)2\pi\omega_{\Gamma_{E}}(e^{it})}\right|^{p}\omega_{\Gamma_{E}}(e^{it})\,\mathrm{d}t.
\end{multline*}
Using Lukashov's inequality \eqref{Luk} we obtain that

\begin{multline*}
\le\int_{H}\left|\frac{T_{n}'(t)}{n2\pi\omega_{\Gamma_{E}}(e^{it})}\left(\frac{n}{\deg\left(T_{n}q\right)}q\left(t\right)-1\right)\right|^{p}\omega_{\Gamma_{E}}(e^{it})\,\mathrm{d}t\\
+\int_{H}\left|\frac{T_{n}(t)\deg\left(q\right)2\pi\omega_{\Gamma_{E}}\left(e^{it}\right)\left\Vert q\right\Vert _{E}}{\deg\left(T_{n}q\right)2\pi\omega_{\Gamma_{E}}(e^{it})}\right|^{p}\omega_{\Gamma_{E}}(e^{it})\,\mathrm{d}t\\
\le A\left(T_{n},H\right)\left(F_{n}^{p}+\left|\frac{n}{n+3n^{2\theta}}-1\right|^{p}\right)+B\left(T_{n},H\right)\left(\frac{\deg q}{\deg\left(T_{n}q\right)}\right)^{p}\\
\le A\left(T_{n},E\right)\left(F_{n}^{p}+\left(\frac{3n^{2\theta}}{n}\right)^{p}\right)+B\left(T_{n},E\right)\left(\frac{3n^{2\theta}}{n}\right)^{p}.
\end{multline*}
This way we established \eqref{eq:3}.

For \eqref{eq:3b},

\begin{multline*}
A(T_{n}q,X)\le\int_{X}\left|\frac{T_{n}'(t)q(t)}{\deg\left(T_{n}q\right)2\pi\omega_{\Gamma_{E}}(e^{it})}\right|^{p}\omega_{\Gamma_{E}}(e^{it})\,\mathrm{d}t\\
+\int_{X}\left|\frac{T_{n}\left(t\right)q'\left(t\right)}{\deg\left(T_{n}q\right)2\pi\omega_{\Gamma_{E}}(e^{it})}\right|^{p}\omega_{\Gamma_{E}}(e^{it})\,\mathrm{d}t
\end{multline*}
Apply \eqref{qnorm1} to the first integral and Lukashov's inequality
\eqref{Luk} to the second one to continue the inequality as

\begin{multline*}
\le\left(\frac{n}{\deg\left(T_{n}q\right)}\right)^{p}A\left(T_{n},X\right)+\left(\frac{\deg\left(q\right)}{\deg\left(T_{n}q\right)}\right)^{p}B(T_{n},X)\\
\le A\left(T_{n},X\right)+\left(\frac{3n^{2\theta}}{n}\right)^{p}B\left(T_{n},X\right)
\end{multline*}
where we again use \eqref{eq:uj_felt}. This proves \eqref{eq:3b}.

\eqref{eq:4} is immediate consequence of \eqref{qnorm1}.

Before the verification of \eqref{eq:5} and \eqref{eq:6}, recall
that $H$ is an interval of the second type. Hence, for sufficiently
large $n$, $H$ is subset of the union in \eqref{set:core}. This
comes from property \eqref{eq:III} and that $H$ contains precisely
one inner extremal point. For \eqref{eq:5}, we have

\[
\left|A_{\delta}\left(T_{n}q,E\left(\delta\right)\right)-A_{\delta}\left(T_{n}q,H\right)\right|=A_{\delta}\left(T_{n}q,H_{b}\right)+A_{\delta}\left(T_{n}q,E\left(\delta\right)\setminus\left(H\cup H_{b}\right)\right).
\]
First we estimate $A_{\delta}\left(T_{n}q,H_{b}\right)$. \eqref{ineq:2}
shows that

\begin{multline*}
A_{\delta}\left(T_{n}q,H_{b}\right)=\int_{H_{b}}\left|\frac{\left(T_{n}\left(t\right)q\left(t\right)\right)'}{\deg\left(T_{n}q\right)2\pi\omega_{\Gamma_{E\left(\delta\right)}}(e^{it})}\right|^{p}\omega_{\Gamma_{E\left(\delta\right)}}(e^{it})\,\mathrm{d}t\\
\le\int_{H_{b}}\left|\frac{T_{n}'\left(t\right)q\left(t\right)}{\deg\left(T_{n}q\right)2\pi\omega_{\Gamma_{E\left(\delta\right)}}(e^{it})}\right|^{p}\omega_{\Gamma_{E\left(\delta\right)}}(e^{it})\,\mathrm{d}t\\
+\int_{H_{b}}\left|\frac{T_{n}\left(t\right)q'\left(t\right)}{\deg\left(T_{n}q\right)2\pi\omega_{\Gamma_{E\left(\delta\right)}}(e^{it})}\right|^{p}\omega_{\Gamma_{E\left(\delta\right)}}(e^{it})\,\mathrm{d}t
\end{multline*}

First employ \eqref{w/w} to both integrals to replace $\omega_{\Gamma_{E\left(\delta\right)}}$
with $\omega_{\Gamma_{E}}$ and then we apply Lukashov's inequality
\eqref{Luk}. Then the inequality is continued as

\begin{multline*}
\le D_{\delta}^{1-p}\int_{H_{b}}\left|\frac{T_{n}'\left(t\right)q\left(t\right)}{\deg\left(T_{n}q\right)2\pi\omega_{\Gamma_{E}}(e^{it})}\right|^{p}\omega_{\Gamma_{E}}(e^{it})\,\mathrm{d}t\\
+D_{\delta}^{1-p}\int_{H_{b}}\left|\frac{T_{n}\left(t\right)q'\left(t\right)}{\deg\left(T_{n}q\right)2\pi\omega_{\Gamma_{E}}(e^{it})}\right|^{p}\omega_{\Gamma_{E}}(e^{it})\,\mathrm{d}t\\
\le D_{\delta}^{1-p}\left(\frac{n}{\deg\left(T_{n}q\right)}\right)^{p}a\left(T_{n},H_{b}\right)A\left(T_{n},E\right)\\
+D_{\delta}^{1-p}\left(\frac{\deg\left(q\right)}{\deg\left(T_{n}q\right)}\right)^{p}b\left(T_{n},H_{b}\right)B\left(T_{n},E\right)\\
\le D_{\delta}^{1-p}a\left(T_{n},H_{b}\right)A\left(T_{n},E\right)+D_{\delta}^{1-p}3^{p}\left(\frac{n^{2\theta}}{n}\right)^{p}b\left(T_{n},H_{b}\right)B\left(T_{n},E\right).
\end{multline*}

Now we turn to $A_{\delta}\bigl(T_{n}q,E\left(\delta\right)\setminus(H\cup H_{b})\bigr)$.
We may assume that $||T_{n}||_{E}=1$. \eqref{ineq:2} gives that

\begin{multline*}
A_{\delta}\left(T_{n}q,E\left(\delta\right)\setminus\left(H\cup H_{b}\right)\right)\\
\le\int_{E\left(\delta\right)\setminus(H\cup H_{b})}\left|\frac{T_{n}'(t)q(t)}{\deg(T_{n}q)2\pi\omega_{\Gamma_{E\left(\delta\right)}}(e^{it})}\right|^{p}\omega_{\Gamma_{E\left(\delta\right)}}(e^{it})\,\mathrm{d}t\\
+\int_{E\left(\delta\right)\setminus(H\cup H_{b})}\left|\frac{T_{n}(t)q'(t)}{\deg(T_{n}q)2\pi\omega_{\Gamma_{E\left(\delta\right)}}(e^{it})}\right|^{p}\omega_{\Gamma_{E\left(\delta\right)}}(e^{it})\,\mathrm{d}t.
\end{multline*}
By \eqref{w/w}, \eqref{approx1} and Lukashov's inequality \eqref{Luk}
we can continue as

\begin{multline*}
\le D_{\delta}^{1-p}F_{n}^{p}\int_{E\setminus(H\cup H_{b})}\left|\frac{n2\pi\left\Vert T_{n}\right\Vert _{E\left(\delta\right)}\omega_{\Gamma_{E\left(\delta\right)}}(e^{it})}{\deg\left(T_{n}q\right)2\pi\omega_{\Gamma_{E\left(\delta\right)}}(e^{it})}\right|^{p}\omega_{\Gamma_{E\left(\delta\right)}}(e^{it})\,\mathrm{d}t\\
+D_{\delta}^{1-p}\int_{E\setminus(H\cup H_{b})}\left|\frac{T_{n}\left(t\right)\deg\left(q\right)2\pi\left\Vert q\right\Vert _{E\left(\delta\right)}\omega_{\Gamma_{E\left(\delta\right)}}(e^{it})}{\deg\left(T_{n}q\right)2\pi\omega_{\Gamma_{E\left(\delta\right)}}(e^{it})}\right|^{p}\omega_{\Gamma_{E\left(\delta\right)}}(e^{it})\,\mathrm{d}t\\
\le D_{\delta}^{1-p}F_{n}^{p}\left(\frac{n}{\deg\left(T_{n}q\right)}\right)^{p}\int_{E\setminus(H\cup H_{b})}1\,\omega_{\Gamma_{E\left(\delta\right)}}(e^{it})\,\mathrm{d}t\\
+D_{\delta}^{1-p}\left(\frac{\deg\left(q\right)}{\deg\left(T_{n}q\right)}\right)^{p}\int_{E\setminus(H\cup H_{b})}\left|T_{n}\left(t\right)\right|^{p}\omega_{\Gamma_{E\left(\delta\right)}}(e^{it})\,\mathrm{d}t\\
\le D_{\delta}^{1-p}F_{n}^{p}\left(\frac{n}{\deg\left(T_{n}q\right)}\right)^{p}+D_{\delta}^{1-p}\left(\frac{\deg\left(q\right)}{\deg\left(T_{n}q\right)}\right)^{p}B\left(T_{n},E\right).
\end{multline*}

The last expression is less than $o(1)B(T_{n},E),$ if we consider
that $F_{n}^{p/2}\le B(T_{n},E)$ where $n$ is large enough. This
is implied by Lemma \ref{also} as follows. Assume that $x_{0}\in E$
is a point for which $\left|T_{n}\left(x_{0}\right)\right|=\left\Vert T_{n}\right\Vert _{E}=1.$
Then $x_{0}$ is also an element of a component $\left[v_{2l-1},v_{2l}\right]$
for some $l.$ Now by Lemma \ref{also} we have
\[
\int_{v_{2l-1}}^{v_{2l}}\left|T_{n}(t)\right|^{p}\omega_{\Gamma_{E}}(t)\,\mathrm{d}t\ge\conxviii\frac{1}{2^{p}}\frac{1}{n^{2}}.
\]
If $n$ is large enough depending on $\conxviii$, $p$ (and $F_{n}$),
then we can write 
\begin{equation}
\conxviii\frac{1}{2^{p}}\frac{1}{n^{2}}\ge F_{n}^{p/2}=\conxvi^{p/2}\exp\left(-\coni n^{\theta}\cdot\frac{p}{2}\right).\label{eq:subpol}
\end{equation}
where the second inequality holds if $n$ is large enough. Using these,
we can continue the estimate
\[
D_{\delta}^{1-p}F_{n}^{p}\left(\frac{n}{\deg\left(T_{n}q\right)}\right)^{p}\le D_{\delta}^{1-p}F_{n}^{p/2}F_{n}^{p/2}\le D_{\delta}^{1-p}F_{n}^{p/2}B\left(T_{n},E\right)
\]
and 
\[
D_{\delta}^{1-p}\left(\frac{\deg\left(q\right)}{\deg\left(T_{n}q\right)}\right)^{p}B\left(T_{n},E\right)\le D_{\delta}^{1-p}3^{p}\left(\frac{n^{2\theta}}{n}\right)^{p}B\left(T_{n},E\right).
\]
Therefore
\[
A_{\delta}\left(T_{n}q,E\left(\delta\right)\setminus\left(H\cup H_{b}\right)\right)\le D_{\delta}^{1-p}\left(F_{n}^{p/2}+3^{p}\left(\frac{n^{2\theta}}{n}\right)^{p}\right)B\left(T_{n},E\right).
\]
Finally, summing up, we can write
\begin{multline*}
\left|A_{\delta}\left(T_{n}q,E\left(\delta\right)\right)-A_{\delta}\left(T_{n}q,H\right)\right|\le D_{\delta}^{1-p}a\left(T_{n},H_{b}\right)A\left(T_{n},E\right)\\
+\left(D_{\delta}^{1-p}F_{n}^{p/2}+D_{\delta}^{1-p}3^{p}\left(\frac{n^{2\theta}}{n}\right)^{p}+\left(\frac{n^{2\theta}}{n}\right)^{p}b\left(T_{n},H_{b}\right)\right)B\left(T_{n},E\right).
\end{multline*}

The proof of \eqref{eq:6} is similar to that of \eqref{eq:5} but
it is simpler. First
\[
\left|B_{\delta}\left(T_{n}q,E\left(\delta\right)\right)-B_{\delta}\left(T_{n}q,H\right)\right|=B_{\delta}\left(T_{n}q,H_{b}\right)+B_{\delta}\left(T_{n}q,E\left(\delta\right)\setminus\left(H\cup H_{b}\right)\right).
\]
Then
\begin{multline*}
B_{\delta}\left(T_{n}q,H_{b}\right)=\int_{H_{b}}\left|T_{n}\left(t\right)q\left(t\right)\right|^{p}\omega_{\Gamma_{E\left(\delta\right)}}\left(e^{it}\right)\,\mathrm{d}t\\
\le D_{\delta}\int_{H_{b}}\left|T_{n}\left(t\right)q\left(t\right)\right|^{p}\omega_{\Gamma_{E}}\left(e^{it}\right)\,\mathrm{d}t\\
\le D_{\delta}b\left(T_{n},H_{b}\right)B\left(T_{n},E\right)
\end{multline*}
and
\begin{multline*}
B_{\delta}\left(T_{n}q,E\left(\delta\right)\setminus\left(H\cup H_{b}\right)\right)=\int_{E\left(\delta\right)\setminus\left(H\cup H_{b}\right)}\left|T_{n}\left(t\right)q\left(t\right)\right|^{p}\omega_{\Gamma_{E\left(\delta\right)}}\left(e^{it}\right)\,\mathrm{d}t\\
\le D_{\delta}F_{n}^{p}B_{\delta}\left(T_{n},E\left(\delta\right)\setminus\left(H\cup H_{b}\right)\right)\le D_{\delta}F_{n}^{p}B\left(T_{n},E\right).
\end{multline*}
Finally, summing up, we can write 
\[
\left|B_{\delta}\left(T_{n}q,E\left(\delta\right)\right)-B_{\delta}\left(T_{n}q,H\right)\right|\le D_{\delta}\left(F_{n}^{p}+b\left(T_{n},H_{b}\right)\right)B\left(T_{n},E\right).
\]

\end{proof}

\subsection{Application of lemmas}

Here we apply the previous lemmas (see \cite{totikvarga}, pp. 401-416
and see that of \cite{nagytookos}) and investigate the occurring
error terms in detail.

\subsection{Reviewing the first and third cases\label{sub:first_case}}

First, we consider Subsection 5.3 from \cite{nagytookos} and Subsection
3.3 from \cite{totikvarga}. This setting is called first and third
cases. Recall that the set $H\subset E$ satisfies \eqref{eq:I},
\eqref{eq:II-a} and \eqref{eq:II-b}. We follow those steps, see
(30) in \cite{nagytookos} and (17) in \cite{totikvarga}. Therefore,
with Lemmas \ref{lem:7} and \ref{lem:8}, we can write
\begin{multline*}
\left(2N\right)A\left(T_{n},H\right)\le\left(2N\right)\left(\frac{\deg\left(T_{n}q\right)}{n}\right)^{p}B\left(T_{n},H\right)\\
+2N\left(4F_{n}^{p}+a\left(T_{n},H_{b}\right)\right)A\left(T_{n},E\right)\\
+2N\left(4\cdot3^{p}\left(\frac{n^{2\theta}}{n}\right)^{p}+3\left(\frac{n^{*}}{n}\right)^{p}F_{n}^{p}+\left(\frac{n^{*}}{n}\right)^{p}b\left(T_{n},H_{b}\right)\right)B\left(T_{n},E\right)
\end{multline*}
We can rewrite it with the help of 
\[
\left(\frac{\deg\left(T_{n}q\right)}{n}\right)^{p}-1=\left(1+\frac{\deg q}{n}\right)^{p}-1\le\left(1+\frac{3n^{2\theta}}{n}\right)^{p}-1\le p\frac{3n^{2\theta}}{n}
\]
where we used $0<p<1$.

Here, using Lemma \ref{lem:7}, \eqref{est:fp_becsles} and \eqref{eq:II-a},
the error term of $A\left(T_{n},E\right)$ can be estimated as follows
\[
4F_{n}^{p}+a\left(T_{n},H_{b}\right)\le\frac{4\conxix+1}{n^{\gamma}}
\]
and, for the error term of $B\left(T_{n},E\right)$, 
\begin{multline*}
4\cdot3^{p}\left(\frac{n^{2\theta}}{n}\right)^{p}+\left(\frac{n^{*}}{n}\right)^{p}F_{n}^{p}+\left(\frac{n^{*}}{n}\right)^{p}b\left(T_{n},H_{b}\right)+p\frac{3n^{2\theta}}{n}\\
\le4\cdot3^{p}\frac{1}{n^{\gamma}}+3\cdot2^{p}\conxix\frac{1}{n^{\gamma}}+2^{p}\frac{1}{n^{\gamma}}+3p\left(\frac{n^{2\theta}}{n}\right)^{p}\\
\le\left(4\cdot3^{p}+3\cdot2^{p}\conxix+2^{p}+3p\right)\frac{1}{n^{\gamma}}.
\end{multline*}
Therefore, with $\conxxv:=\max\left(4\conxix+1,\ 4\cdot3^{p}+3\cdot2^{p}\conxix+2^{p}+3p\right)$
we can write
\begin{equation}
A\left(T_{n},H\right)\le B\left(T_{n},H\right)+\frac{\conxxv}{n^{\gamma}}A\left(T_{n},E\right)+\frac{\conxxv}{n^{\gamma}}B\left(T_{n},E\right).\label{FIRST_CASE}
\end{equation}

\subsection{Reviewing the second case\label{sub:second_case}}

As for the second case (see Subsection 5.4 in \cite{nagytookos} and
Subsection 3.4 in \cite{totikvarga}), we have to track carefully
the error terms.

Consider an inner extremal $\zeta_{k,j}\in E$ where $k\in\left\{ 1,2,\ldots,m\right\} $
and $j\in\left\{ 1,\ldots,j_{k}\right\} $ are fixed. 

Let $H$ be an interval of second type containing $\zeta_{k,j}$.
Then, by property \eqref{eq:III}, $\left|H\right|\le4n^{\gamma-\kappa}$.
So if $\delta$ is fixed and $n$ is large enough, then $\left|\zeta_{k,j}-\zeta_{k,j}\left(\delta\right)\right|\ge\left|H\right|$.
Therefore $H$ contains no inner extremals of $U_{\delta}$ and it
behaves as an integral of first type with respect to $E\left(\delta\right)=U_{\delta}^{-1}\left[-1,1\right]$.

We need the following notations and observations. Let $\conxxvi:=\max\left(\conxix+3^{p},3+3^{p}\right)$.
Since $0<p\le1$, and $D_{\delta}\ge1$, we have $D_{\delta}^{1-p}\le D_{\delta}$.
We use \eqref{eq:subpol} and $0\le\gamma<1$, so if $n$ is large,
then 
\[
F_{n}^{p/2}\le n^{-\gamma}.
\]
The error terms in Lemma \ref{lem:11} have polynomial decay in the
following sense. We use \eqref{est:fp_becsles}, \eqref{eq:uj_felt}:
\begin{equation}
F_{n}^{p}+3^{p}\left(\frac{n^{2\theta}}{n}\right)^{p}\le\conxix n^{-\gamma}+3^{p}n^{-\gamma}\le\frac{\conxxvi}{n^{\gamma}},\label{eq:3hibatag1}
\end{equation}
\begin{equation}
3^{p}\left(\frac{n^{2\theta}}{n}\right)^{p}\le\frac{\conxxvi}{n^{\gamma}},\label{eq:3hibatag2}
\end{equation}
and we also use \eqref{eq:II-a}, \eqref{eq:II-b}, so
\begin{equation}
D_{\delta}^{1-p}a\left(T_{n},H_{b}\right)\le D_{\delta}^{1-p}n^{-\gamma}\le\frac{D_{\delta}}{n^{\gamma}},\label{eq:5hibatag1}
\end{equation}
\begin{multline}
D_{\delta}^{1-p}F_{n}^{p/2}+D_{\delta}^{1-p}3^{p}\left(\frac{n^{2\theta}}{n}\right)^{p}+\left(\frac{n^{2\theta}}{n}\right)^{p}b\left(T_{n},H_{b}\right)\le D_{\delta}^{1-p}\left(1+3^{p}\right)\frac{1}{n^{\gamma}}+\frac{1}{n^{\gamma}}\\
\le\frac{D_{\delta}\conxxvi}{n^{\gamma}}\label{eq:5hibatag2}
\end{multline}
and with \eqref{est:fp_becsles} 
\begin{equation}
D_{\delta}\left(F_{n}^{p}+b\left(T_{n},H_{b}\right)\right)\le\frac{D_{\delta}\conxxvi}{n^{\gamma}}.\label{eq:6hibatag}
\end{equation}

Furthermore, by Lemma \ref{sulykonv} and $0<p<1$, we have the following
estimate for $t\in H$ 
\begin{equation}
\omega_{\Gamma_{E}}^{1-p}\left(e^{it}\right)\le\omega_{\Gamma_{E\left(\delta\right)}}^{1-p}\left(e^{it}\right)\le D_{\delta}\omega_{\Gamma_{E}}^{1-p}\left(e^{it}\right).\label{est:omega_power}
\end{equation}

Now we start the estimate, as in pp. 147--149 in \cite{nagytookos}.
Using \eqref{eq:3} and \eqref{est:omega_power}, 
\begin{multline*}
A\left(T_{n},H\right)\le A\left(T_{n}q,H\right)+\frac{\conxxvi}{n^{\gamma}}A\left(T_{n},E\right)+\frac{\conxxvi}{n^{\gamma}}B\left(T_{n},E\right)\\
\le A_{\delta}\left(T_{n}q,H\right)+\frac{\conxxvi}{n^{\gamma}}A\left(T_{n},E\right)+\frac{\conxxvi}{n^{\gamma}}B\left(T_{n},E\right).
\end{multline*}

We apply the first case for the polynomial $T_{n}q$ on the interval
$H$ with respect to $E\left(\delta\right)$ (see \eqref{FIRST_CASE}),
and use that $\deg T_{n}q\ge n$, so
\begin{equation}
A_{\delta}\left(T_{n}q,H\right)\le B_{\delta}\left(T_{n}q,H\right)+\frac{\conxxv}{n^{\gamma}}A_{\delta}\left(T_{n}q,E\left(\delta\right)\right)+\frac{\conxxv}{n^{\gamma}}B_{\delta}\left(T_{n}q,E\left(\delta\right)\right)\label{eq:case_one_applied}
\end{equation}
where $\conxxv$ depends on $p$ and $\conxix$ only, $\conxix$ depends
on $\coni$ and $\conxvi$ (and $\gamma$, $\theta$) only, $\coni$
and $\conxvi$ depend on $\theta$ and $\kappa$ only. Therefore,
$\conxxv$ is independent of $\delta$. 

We estimate the error term containing $A_{\delta}\left(.,.\right)$
using \eqref{eq:5} (with the estimates \eqref{eq:5hibatag1}, \eqref{eq:5hibatag2}),
\eqref{est:omega_power} and \eqref{eq:3} (with the \eqref{eq:3hibatag1},
\eqref{eq:3hibatag2}) as follows
\begin{multline*}
A_{\delta}\left(T_{n}q,E\left(\delta\right)\right)\le A_{\delta}\left(T_{n}q,H\right)+\frac{D_{\delta}}{n^{\gamma}}A\left(T_{n},E\right)+\frac{D_{\delta}\conxxvi}{n^{\gamma}}B\left(T_{n},E\right)\\
\le D_{\delta}A\left(T_{n}q,H\right)+\frac{D_{\delta}}{n^{\gamma}}A\left(T_{n},E\right)+\frac{D_{\delta}\conxxvi}{n^{\gamma}}B\left(T_{n},E\right)\\
\le D_{\delta}A\left(T_{n},H\right)+\left(\frac{D_{\delta}}{n^{\gamma}}+\frac{D_{\delta}\conxxvi}{n^{\gamma}}\right)A\left(T_{n},E\right)+\frac{2D_{\delta}\conxxvi}{n^{\gamma}}B\left(T_{n},E\right)\\
\le D_{\delta}\left(1+\frac{\conxxvi+1}{n^{\gamma}}\right)A\left(T_{n},E\right)+\frac{2D_{\delta}\conxxvi}{n^{\gamma}}B\left(T_{n},E\right)
\end{multline*}

We estimate the error term with $B_{\delta}\left(.,.\right)$ with
\eqref{eq:6} (and \eqref{eq:6hibatag}) and \eqref{eq:4} as follows
\begin{multline*}
B_{\delta}\left(T_{n}q,E\left(\delta\right)\right)\le B_{\delta}\left(T_{n}q,H\right)+\frac{D_{\delta}\conxxvi}{n^{\gamma}}B\left(T_{n},E\right)\\
\le D_{\delta}B\left(T_{n}q,H\right)+\frac{D_{\delta}\conxxvi}{n^{\gamma}}B\left(T_{n},E\right)\le D_{\delta}B\left(T_{n},H\right)+\frac{D_{\delta}\conxxvi}{n^{\gamma}}B\left(T_{n},E\right)\\
\le D_{\delta}\left(1+\frac{\conxxvi}{n^{\gamma}}\right)B\left(T_{n},E\right).
\end{multline*}

We estimate the $B_{\delta}$ on the right hand side of \eqref{eq:case_one_applied}
as follows
\[
B_{\delta}\left(T_{n}q,H\right)\le D_{\delta}B\left(T_{n}q,H\right)\le D_{\delta}B\left(T_{n},H\right).
\]
Collecting these estimates together, we can write
\begin{multline*}
A\left(T_{n},H\right)\le D_{\delta}B\left(T_{n},H\right)\\
+A\left(T_{n},E\right)\left(\frac{\conxxv}{n^{\gamma}}D_{\delta}\left(1+\frac{1+\conxxvi}{n^{\gamma}}\right)+\frac{\conxxvi}{n^{\gamma}}\right)\\
+B\left(T_{n},E\right)\left(\frac{\conxxvi}{n^{\gamma}}+\frac{2\conxxv\conxxvi D_{\delta}}{n^{2\gamma}}+\frac{\conxxv}{n^{\gamma}}D_{\delta}\left(1+\frac{1+\conxxvi}{n^{\gamma}}\right)\right).
\end{multline*}
Here, we estimate the coefficients of the error terms as follows with
$\conxxvii:=\max\left(\conxxv\left(2+\conxxvi\right)+\conxxvi,\conxxvi+3\conxxv\conxxvi+2\conxxv,\conxxv\right)$
\begin{multline*}
\frac{\conxxv}{n^{\gamma}}D_{\delta}\left(1+\frac{1+\conxxvi}{n^{\gamma}}\right)+\frac{\conxxvi}{n^{\gamma}}\le\frac{\conxxv}{n^{\gamma}}D_{\delta}\left(1+1+\conxxvi\right)+\frac{\conxxvi}{n^{\gamma}}\\
\le\frac{D_{\delta}}{n^{\gamma}}\left(\conxxv\left(2+\conxxvi\right)+\conxxvi\right)\le\frac{D_{\delta}\conxxvii}{n^{\gamma}}
\end{multline*}
and
\begin{multline*}
\frac{\conxxvi}{n^{\gamma}}+\frac{2\conxxv\conxxvi D_{\delta}}{n^{2\gamma}}+\frac{\conxxv}{n^{\gamma}}D_{\delta}\left(1+\frac{1+\conxxvi}{n^{\gamma}}\right)\le\frac{\conxxvi}{n^{\gamma}}+\frac{2\conxxv\conxxvi D_{\delta}}{n^{\gamma}}\\
+\frac{\conxxv}{n^{\gamma}}D_{\delta}\left(1+1+\conxxvi\right)\le\frac{D_{\delta}}{n^{\gamma}}\left(\conxxvi+2\conxxv\conxxvi+\conxxv\left(2+\conxxvi\right)\right)\le\frac{D_{\delta}\conxxvii}{n^{\gamma}}.
\end{multline*}
Therefore, we have the following estimate
\begin{equation}
A\left(T_{n},H\right)\le D_{\delta}B\left(T_{n},H\right)+\frac{D_{\delta}\conxxvii}{n^{\gamma}}A\left(T_{n},E\right)+\frac{D_{\delta}\conxxvii}{n^{\gamma}}B\left(T_{n},E\right).\label{SECOND_CASE}
\end{equation}
Note that $ $$\conxxvii$ depends on $\conxxv$ and $\conxxvi$ only,
and these two constants are independent of $\delta$.

\subsection{Proving Theorem \ref{thm:arestov} for T-sets, then for union of
finitely many intervals and Theorem \ref{thm:sharpness}}

In this section we finish the proof of Theorem \ref{thm:arestov},
following essentially Sections 5.5 and 6 in \cite{nagytookos}. With
Sections \ref{sub:first_case} and \ref{sub:second_case} at hand,
we sum up the results for those intervals, copying the steps in Section
5.5 in \cite{nagytookos}. This way, using $D_{\delta}\ge1$ and $\conxxvii\ge\conxxv$,
we can write 
\[
A\left(T_{n},E\right)\le D_{\delta}B\left(T_{n},E\right)+12N\frac{C_{29}D_{\delta}}{n^{\gamma}}A\left(T_{n},E\right)+12N\frac{C_{29}D_{\delta}}{n^{\gamma}}B\left(T_{n},E\right).
\]
Hence, Theorem \ref{thm:arestov} is proved for T-sets. 

For sets consisting of union of finitely many intervals, Section 6
in \cite{nagytookos} can be applied mutatis mutandis.

As regards Theorem \ref{thm:sharpness}, the Section 7 in \cite{nagytookos}
with the simple cosine substitution gives the proof, since the authors
did not use there that the power is bigger than $1$.

\section*{Acknowledgement}

Béla Nagy was supported by Magyary scholarship: This research was
realized in the frames of TÁMOP 4.2.4. A/2-11-1-2012-0001 „National
Excellence Program – Elaborating and operating an inland student and
researcher personal support system.” The project was subsidized by
the European Union and co-financed by the European Social Fund.

The second author, Tamás Varga was supported by the European Research
Council Advanced grant No. 267055, while he had a position at the
Bolyai Institute, University of Szeged, Aradi v. tere 1, Szeged 6720,
Hungary.


\providecommand{\bysame}{\leavevmode\hbox to3em{\hrulefill}\thinspace}
\providecommand{\MR}{\relax\ifhmode\unskip\space\fi MR }
\providecommand{\MRhref}[2]{%
  \href{http://www.ams.org/mathscinet-getitem?mr=#1}{#2}
}
\providecommand{\href}[2]{#2}

\medskip{}

Béla Nagy,

MTA-SZTE Analysis and Stochastics Research Group, Bolyai Institute,
Szeged, H-6720, Aradi vértanúk tere 1, Hungary,

email: \href{mailto:nbela@math.u-szeged.hu}{nbela@math.u-szeged.hu}

\smallskip{}

Tamás Varga,

Bolyai Institute, University of Szeged, Szeged, H-6720, Aradi vértanúk
tere 1, Hungary,

email: \href{mailto:vargata@math.u-szeged.hu}{vargata@math.u-szeged.hu}
\end{document}